\documentclass[12pt]{amsart}
\usepackage{amsfonts,amssymb,amscd,amsmath,enumerate,verbatim}
\usepackage[latin1]{inputenc}
\usepackage{amscd}
\usepackage{latexsym}
\usepackage{ytableau}

\usepackage[dvipdfmx]{graphicx}
\usepackage{mathptmx}

\def\NZQ{\mathbb}               

\def\ZZ{{\NZQ Z}}
\def\RR{{\NZQ R}}

\def\KK{{\NZQ K}}

\def\frk{\mathfrak}               

\def\Sf{{\frk S}}

\def\Phi{{\frk N}}
\def\ab{{\mathbf a}}

\def\xb{{\mathbf x}}

\def\Ib{{\mathbf I}}
\def\Pb{{\mathbf P}}

\def\opn#1#2{\def#1{\operatorname{#2}}} 
\opn\ini{in} \opn\sgn{sgn} \opn\Gr{Gr} \opn\Im{Im}
\opn\gr{gr}

\def\Ac{{\mathcal A}}

\def\id{{\textnormal{id}}}

\newtheorem{Theorem}{Theorem}[section]
\newtheorem{Lemma}[Theorem]{Lemma}

\newtheorem{Proposition}[Theorem]{Proposition}

\theoremstyle{definition}
\newtheorem{Remark}[Theorem]{Remark}

\newtheorem{Example}[Theorem]{Example}

\newtheorem{Definition}[Theorem]{Definition}

\newtheorem{Question}[Theorem]{Question}

\let\epsilon\varepsilon
\let\phi=\varphi
\let\kappa=\varkappa
\opn\dis{dis}
\opn\height{height}
\opn\dist{dist}
\def\pnt{{\raise0.5mm\hbox{\large\bf.}}}

\opn\Lex{Lex}
\opn\conv{conv}

\begin{document}

\title{New kinds of block diagonal matching fields and toric degenerations of Grassmannians}

\author{Akihiro Higashitani and Nobukazu Kowaki}

\address{Akihiro Higashitani,
Department of Pure and Applied Mathematics,
Graduate School of Information Science and Technology,
Osaka University,
Suita, Osaka 565-0871, Japan}
\email{higashitani@ist.osaka-u.ac.jp}

\address{Nobukazu Kowaki
	Department of Pure and Applied Mathematics,
Graduate School of Information Science and Technology,
Osaka University,
Suita, Osaka 565-0871, Japan}
\email{u793177f@ecs.osaka-u.ac.jp}


\begin{abstract}
Block diagonal matching field has many previous works. In general, a coherent matching field induces a monomial order to Pl\"{u}cker algebra, and block diagonal matching fields are a kind of coherent matching fields. 
In the present paper, we introduce a new kind of block diagonal matching fields and study the problem when they give a SAGBI basis. 
As a corollary, we provide a new family of toric degenerations of Grassmannians by using SAGBI bases.
\end{abstract}

\maketitle

\section{Introduction}

A \textit{Grassmannian} $\mathrm{Gr}(r,n)$ is the set of all subspaces of a given dimension $r$ in $\KK^n$, where $\KK$ is a field. 

Toric degenerations provide the tool of polyhedral geometry to study algebraic geometry.
So, we study toric degenerations of Grassmannians. In particular we focus on toric degeneration by using SAGBI basis.

SAGBI basis was independently introduced by Kapur and Madlener \cite{kapur1989completion} and Robbiano and Sweedler \cite{robbiano1990subalgebra}. SAGBI means Subalgebra Analogue of Gr\"{o}bner Bases for Ideals.  SAGBI basis answers the subalgebra membership problem similar to Gr\"{o}bner bases to the ideal membership problem. What is important for us is that toric degenerations of Grassmannians by matching fields are characterized by that Pl\"{u}cker coordinates form SAGBI basis.

We need a monomial ordering to use SAGBI basis. To give this, we use matching fields. Given positive integers $r$ and $n$, a \textit{matching field} denoted by $\Lambda$ when there is no confusion, is a choice of a permutation $\Lambda(I) \in \mathfrak{S}_r$ for each $I \in \mathbf{I}_{r, n} = \{ I \subset [n] : |I| = r \}$, where $[n]=\{1,2,\ldots,n\}$ 
. 
A matching field gives rise to a toric degeneration if Pl\"{u}cker coordinates form SAGBI basis for the subalgebra they generate, with respect to a monomial ordering the matching field induces. See Section \ref{sec:matching_fields} for more details.
 Matching fields were born to study the Newton polytope of the product of all maximal minors of a matrix of indeterminates $X=(x_{ij})$, introduced by Sturmfels and Zelevinsky \cite{sturmfels1993maximal}. The most famous example of matching fields is the diagonal matching field. This chooses the diagonal term as the initial monomial for each minor. 
Namely, the diagonal matching field sends every $I \in \mathbf{I}_{r,n}$ to $\mathrm{id} \in \mathfrak{S}_r$.

Mohammadi and Shaw introduced a more generalized class of matching fields, block diagonal matching fields \cite{mohammadi2019toric}. The image of this matching field is $\{ \mathrm{id},\ (1\ 2) \}$. Block diagonal matching fields were studied in \cite{clarke2021combinatorial,clarke2021toric,clarke2024toric,higashitani2022quadratic}. Note that all these studies are  about $(\ab,2)$-block diagonal matching fields in the sense of Definition \ref{def:block}. 

It is a natural question whether $(\ab,\ell)$-block diagonal matching fields give rise to toric degenerations. For the precise definition of $(\ab,\ell)$-block diagonal matching field, see Section \ref{subsec:BDMF}. 

We prove the following main results of the present paper. 
First, we claim the case of $\ell \geq 4$. 
\begin{Theorem}[{See Theorems~\ref{thm:onlyif}, \ref{thm:if} and Lemma~\ref{cor:if}}]\label{thm:main}
Let $\ab=(a_1,\ldots,a_s) \in \ZZ_{>0}^s$ satisfying $\sum_{i=1}^s a_i=n$ and $a_s \geq 2$, and let $\ell \geq 4$. 
Consider the $(\ab,\ell)$-block diagonal matching field $\Lambda_{\ab,\ell}$. 
Then the generating set $\{\det(\xb_I) : I \in \Ib_{r,n}\}$ of the Pl\"ucker algebra $\Ac_{r,n}$ forms a SAGBI basis for $\Ac_{r,n}$ 
with respect to $\Lambda_{\ab,\ell}$ if 
\begin{itemize}
\item $a_1 \leq 3$ holds, and 
\item $r \geq \sum_{u=i}^s a_u$ is satisfied for any $i$ with $2 \leq i \leq s-1$ and $a_i \geq 3$. 
\end{itemize}

On the other hand, the generating set $\{\det(\xb_I) : I \in \Ib_{r,n}\}$ of the Pl\"ucker algebra $\Ac_{r,n}$ does not form a SAGBI basis for $\Ac_{r,n}$ 
with respect to $\Lambda_{\ab,\ell}$ if 
\begin{itemize}
\item $a_1 \geq 5$ holds, or
\item $r+2 \leq \sum_{u=i}^s a_u$ is satisfied for some $i$ with $2 \leq i \leq s-1$ and $a_i \geq 4$. 
\end{itemize}
In particular, when this is the case, $\Lambda_{\ab,\ell}$ gives rise to a toric degeneration of $\Gr(r,n)$. 
\end{Theorem}



We see in detail the condition of $a_1,\cdots a_n$ that is not discussed in the theorem in Subsection \ref{open}.

\vspace{1mm}

A proof of Theorem~\ref{thm:main} consists as follows: 
\begin{itemize}
\item we prove that if $a_1 \geq 5$ or there is $i$ with $2 \leq i \leq s-1$ such that $a_i \geq 4$ and $r+2 \leq \sum_{t=i}^sa_t$, 
then $\{\det(\xb_I) : I \in \Ib_{r,n}\}$ does not form a SAGBI basis (Theorem~\ref{thm:onlyif}); 
\item we prove that if $a_1 \leq 3$ and $a_i \leq 2$ hold for any $i$ with $2 \leq i \leq s-1$, 
then $\{\det(\xb_I) : I \in \Ib_{r,n}\}$ forms a SAGBI basis (Theorem~\ref{thm:if}); 
\item by the same discussion as the proof of Theorem~\ref{thm:if}, we check that 
if $a_1 \leq 3$ holds and $r \geq \sum_{t=i}^s a_t$ is satisfied for any $i$ with $2 \leq i \leq s-1$ such that $a_i \geq 3$, 
then $\{\det(\xb_I) : I \in \Ib_{r,n}\}$ forms a SAGBI basis (Lemma~\ref{cor:if}); 
\end{itemize}

Next, we claim the case of $\ell=3$. 
\begin{Theorem}\label{thm:l=3}
For any $\ab=(a_1,\ldots,a_s) \in \ZZ_{>0}^s$ with $\sum_{i=1}^s a_i=n$, 
the generating set $\{\det(\xb_I) : I \in \Ib_{r,n}\}$ of the Pl\"ucker algebra $\Ac_{r,n}$ forms a SAGBI basis for $\Ac_{r,n}$ with respect to $\Lambda_{\ab,3}$. 
In particular, $\Lambda_{\ab,3}$ gives rise to a toric degeneration of $\Gr(r,n)$. 
\end{Theorem}


The paper is organized as follows: 
Section $2.1$ fixes notations for the Pl\"{u}cker algebras and SAGBI basis. In Section $2.2$ and $2.3$, we recall matching fields and the associated tableaux. In Section $3$, we introduce the block diagonal matching fields and the tools to prove our main theorems. In the last two sections, our main theorem is proved with many examples. 

\section*{Acknowledgements}
The first named author is partially supported by JSPS KAKENHI Grant Number JP24K00521 and JP21KK0043.

\medskip


\section{Preliminaries}

\subsection{Pl\"ucker algebra and SAGBI basis}
Given integers $r$ and $n$ with $1 < r < n$, let $\Ib_{r,n}$ be the collection of all $r$-subsets of $[n]=\{1,2,\ldots,n\}$. 
We often use, e.g., $123 \in \Ib_{3,n}$ to notate $\{1,2,3\}$ for short. 

Let $S = \KK[P_I : I \in  \Ib_{r,n}]$ be the polynomial ring with $\binom{n}{r}$ variables. 
Let $\xb=(x_{ij})_{1 \leq i \leq r, 1 \leq j \leq n}$ be the $r \times n$ matrix of variables 
and let $R=\KK[\xb]$ be the polynomial ring with $rn$ variables. 
The {\em Pl\"ucker ideal} $I_{r,n}$ is defined by the kernel of the ring homomorphism 
$$\psi:S \rightarrow R, \;\; P_I \mapsto \det(\xb_I),$$
where $\xb_I$ denotes the $r \times r$ submatrix of $\xb$ whose columns are indexed by $I$. 
The {\em Pl\"ucker algebra} $\Ac_{r,n}$ is the image $\Im(\psi)$ of this map, which is isomorphic to $S/I_{r,n}$. 
The Pl\"ucker algebra $\Ac_{r,n}$ is well-known to be a homogeneous coordinate ring of the Pl\"ucker embedding of the Grassmannians $\Gr(r,n)$.

\begin{Definition}[SAGBI bases for the Pl\"ucker algebra]
We say that the generating set $\{\det(\xb_I) : I \in \Ib_{r,n}\} \subset R$ of the Pl\"ucker algebra $\Ac_{r,n}$ 
is a {\em SAGBI basis with respect to a weight matrix $M \in \RR^{r \times n}$} if $\ini_M(\Ac_{r,n})=\KK[\ini_M(\det(\xb_I)) : I \in \Ib_{r,n}]$, 
where $\ini_M(\Ac_{r,n})=K[\ini_M(f) : f \in \Ac_{r,n}]$. 
\end{Definition}

\subsection{Matching fields}\label{sec:matching_fields}

Let ${\mathfrak S}_r$ denote the symmetric group on $[r]$. 
An $r \times n$ {\em matching field} is a map $$\Lambda : \Ib_{r,n} \rightarrow {\mathfrak S}_r.$$ 
For $I=\{i_1,\ldots,i_r\} \in \Ib_{r,n}$ with $1 \leq i_1 < \cdots < i_r \leq n$ and $\sigma \in \Sf_r$, we associate the monomial of $R$ 
$$\xb_{\sigma(I)}:=x_{\sigma(1) i_1} \cdots x_{\sigma(r) i_r}.$$
Given a matching field $\Lambda$, in the case $\sigma=\Lambda(I)$, we write $\xb_{\Lambda(I)}$ instead of $\xb_{\sigma(I)}$. 
We define a ring homomorphism $$\psi_\Lambda : S \rightarrow R, \;\;\psi_\Lambda(P_I) = \sgn(\Lambda(I))\xb_{\Lambda(I)},$$ 
where $\sgn(\sigma)$ denotes the signature of $\sigma \in {\mathfrak S}_r$. 
Then the {\em matching field ideal} $J_\Lambda$ of $\Lambda$ is the kernel of $\psi_\Lambda$.

\begin{Definition}[{Coherent matching fields (\cite[Section 1]{sturmfels1993maximal}, \cite[Definition 2.8]{mohammadi2019toric}})]
A matching field $\Lambda$ is said to be {\em coherent} if there exists an $r \times n$ matrix $M \in \RR^{r \times n}$ with its entries in $\RR$ 
such that for every $I \in \Ib_{r,n}$ the initial form $\ini_M( \det(\xb_I) )$ of $\det(\xb_I)$ with respect to $M$, 
which is the sum of all terms of $\det(\xb_I)$ having the lowest weights, is equal to $\psi_\Lambda(P_I)$. 
In this case, we call $\Lambda$ a coherent matching field {\em induced by $M$}. 
For the matching field ideals, we use the notation $J_M$ instead of $J_\Lambda$ if $\Lambda$ is a coherent matching field induced by $M$. 
\end{Definition}
In the original definition \cite[Section 1]{sturmfels1993maximal} of coherent matching fields, 
the initial form is set to be the sum of all terms having the {\em highest} weights, 
but we usually employ the definition with the sum of all terms having the {\em lowest} weights 
when it is related to the context of tropical geometry, following the convention of our main reference \cite{mohammadi2019toric}. 

\begin{Theorem}[{\cite[Theorem 11.4]{sturmfels1996grobner}}]\label{thm:enough}
The generating set $\{\det(\xb_I) : I \in \Ib_{r,n}\}$ of the Pl\"ucker algebra $\Ac_{r,n}$ 
is a SAGBI basis with respect to a weight matrix $M \in \RR^{r \times n}$ 
if and only if $J_\Lambda=\ini_{w_M}(I_{r,n})$ holds, where $w_M$ is the weight vector on $S$ induced by $M$ 
and $\Lambda$ is the coherent matching field induced by $M$. 
\end{Theorem}

\begin{Remark}\label{rem:permute}
Let $M \in \RR^{r \times n}$ be an $r \times n$ matrix which induces a coherent matching field $\Lambda$. 
Consider the other matrix $M'$ obtained by permuting the columns of $M$ and the associated matching field $\Lambda'$. 
Since any permutation of the column of the matrix corresponds to the permutation of the variables $x_{ij}$ of $R=\KK[\xb]$, 
we can see that the generating set gives rise to a SAGBI basis with respect to $M$ if and only if generating set gives rise to a SAGBI basis with respect to $M'$.
\end{Remark}

\begin{Remark}\label{rem:assume}
In the case of $r=2$, we see that any matching field $\Lambda$ gives rise to a SAGBI basis. 
In fact, without loss of generality, we can assume that the first row of the associated matrix of any coherent matching field is all $0$, 
so the weight with respect to $\Lambda$ is determined only by the second row. 
By permuting the entries in the second row in the descreasing order, we obtain the diagonal matching field, which gives rise to a SAGBI basis. 
Hence, $\Lambda$ also gives rise to a SAGBI basis (see Remark~\ref{rem:permute}). 
Namely, we do not need to discuss the case $r=2$.  

Moreover, it is well known that $\Ac_{r,n}$ is isomorphic to $\Ac_{n-r,n}$. 
Hence, we do not need to discuss the case $r=n-2$, either. 

Therefore, in what follows, we may assume that $3 \leq r \leq n-3$. 
\end{Remark}

\subsection{Matching field tableaux}

For an $r \times n$ coherent matching field $\Lambda$, we provide the description of (the product of) the monomials appearing in $\psi(P_I)$ in terms of tableaux. 
This was originally introduced in \cite{MR4108332} (or essentially in \cite{mohammadi2019toric}). 

Given $I=\{i_1,\ldots,i_r\},J=\{j_1,\ldots,j_r\} \in \Ib_{r,n}$ with $1 \leq i_1 < \cdots < i_r \leq n$, $1 \leq j_1 < \cdots < j_r \leq n$ and $\sigma, \tau \in \Sf_r$, respectively, 
let us associate $\xb_{\sigma(I)}$ and $\xb_{\sigma(I)}\xb_{\tau(J)}$ with the following tableaux $T_{\sigma(I)}$ and $T_{\sigma(I),\tau(J)}$: 
$$\ytableausetup{boxsize=2em}
T_{\sigma(I)}=\begin{ytableau} i_{\sigma(1)} \\ i_{\sigma(2)} \\ \vdots \\ i_{\sigma(r)}\end{ytableau}\quad\text{and}\quad
T_{\sigma(I),\tau(J)}=\begin{ytableau} i_{\sigma(1)} &j_{\tau(1)} \\ i_{\sigma(2)} &j_{\tau(2)} \\ \vdots &\vdots \\ i_{\sigma(r)} &j_{\tau(r)} \end{ytableau}$$ 
The product of more than two monomials is similar. 
We use the notation $T_I$ if $\sigma=\Lambda(I)$ for a given coherent matching field. (The case of more than one monomial is similar.) 

Let $I_1,\ldots,I_m,J_1,\ldots,J_m \in \Ib_{r,n}$ and $\sigma_1,\ldots,\sigma_m,\tau_1,\ldots,\tau_m \in \Sf_r$. 
Then we see that two tableaux $T_{\sigma_1(I_1),\ldots,\sigma_m(I_m)}$ and $T_{\tau_1(J_1),\ldots,\tau_m(J_m)}$ are row-wise equal if and only if 
the corresponding two monomials $\xb_{\sigma_1(I_1)} \cdots \xb_{\sigma_m(I_m)}$ and $\xb_{\tau_1(J_1)} \cdots \xb_{\tau_m(J_m)}$ 
appearing in $\psi(P_{I_1} \cdots P_{I_m})$ and $\psi(P_{J_1} \cdots P_{J_m})$ coincide. 
Namely, this monomial disappears from the polynomial $\psi(P_{I_1} \cdots P_{I_m}-P_{J_1} \cdots P_{J_m})$. 

\begin{Example}
Let $\Lambda$ be a $4 \times 8$ matching field and let $\Lambda(2357)=(1 \ 2 \ 3)$. 
Then $\xb_{\Lambda(2357)}$ corresponds to $T_{2357}=\ytableausetup{boxsize=1em}\begin{ytableau} 3 \\ 5 \\ 2 \\ 7\end{ytableau}$. 
Moreover, if $\Lambda(3456)=\mathrm{id}$, then $T_{2357,3456}=\begin{ytableau} 3 &3 \\ 5 &4 \\ 2 &5 \\ 7 &6\end{ytableau}$. 
\end{Example}

\medskip

\section{$(\ab,\ell)$-block diagonal matching fields}\label{sec:prepare}

\subsection{Definition of $(\ab,\ell)$-block diagonal matching fields}\label{subsec:BDMF}
Inspired by \cite[Definition 4.1]{mohammadi2019toric}, we introduce a new kind of coherent matching fields. 
The main object of the present paper is the following: 
\begin{Definition}[$(\ab,\ell)$-block diagonal matching fields]\label{def:block}
Let $\ab = (a_1,\dots, a_s) \in \ZZ_{>0}^s$ such that $\sum_{i=1}^s a_i = n$ and let $\ell$ be an integer with $2 \leq \ell \leq r$. For $k=1,2,\dots , s$, let 
$$I_k=\{\alpha_{k-1}+1 ,\alpha_{k-1}+ 2,  \ldots, \alpha_k \}=[\alpha_k] \setminus [\alpha_{k-1}],$$
where $\alpha_0=0$ and $\alpha_k= \sum_{i=1}^k a_i$. Note that $\alpha_s =n$.
Then the $(\ab,\ell)$-{\em block diagonal matching field} $\Lambda_{\ab,\ell}$ is a coherent matching field induced by the following matrix: 
\begin{align*}
&M_{\ab,\ell}=\\
&\left(
\begin{array}{ccc|ccc|c|ccc}
0         &  \cdots & 0                                       &0              & \cdots  & 0 &\cdots   &0 &\cdots &0\\
n         &  \cdots &n-\alpha_1+1                             &n-\alpha_1 &  \cdots &n-\alpha_2+1   &\cdots    &n-\alpha_{s-1} &\cdots & 1 \\
 \vdots &            &\vdots                              &\vdots &           &\vdots   &        &\vdots & & \vdots \\
\alpha_1\beta^{\ell-2} &\cdots &\beta^{\ell-2}     &\alpha_2\beta^{\ell-2} &\cdots&(\alpha_1+1)\beta^{\ell-2}  &\cdots 
&\alpha_s\beta^{\ell-2} &\cdots &(\alpha_{s-1}+1)\beta^{\ell-2} \\
 \vdots &            &\vdots                              &\vdots &           &\vdots   &        &\vdots & & \vdots \\
n\beta^{r-2}      &  \cdots &     &(n-\alpha_1)\beta^{r-2} &  \cdots & &\cdots    & &\cdots & \beta^{r-2} 
\end{array} \right),
\end{align*}
where $\beta \gg 0$. 
Note that the case $\ell=2$ is nothing but the $s$-block diagonal matching field defined in \cite[Definition 4.1]{mohammadi2019toric}. 
(The matrix $M_{\ab,2}$ is equal to the one appearing in \cite[Definition 2.2]{higashitani2022quadratic}.) 
\end{Definition}

\begin{Question}
Does every $(\ab,\ell)$-block diagonal matching field give rise to a toric degeneration of $\Gr(r,n)$? 
\end{Question}
Note that this question for $\ell=2$ was mentioned in \cite[Question 5.7]{higashitani2022quadratic}, and partially solved in \cite[Theroem 3]{clarke2024toric}.
The goal of the present paper is to give an almost complete answer of this question for $\ell \geq 3$. 

\subsection{Properties on $\Lambda_{\ab,\ell}$}
Given $\ab=(a_1,\ldots,a_s) \in \ZZ_{>0}^s$ with $\sum_{i=1}^s a_i=n$ and $\ell \geq 2$, 
consider the $(\ab,\ell)$-block diagonal matching field $\Lambda_{\ab,\ell}$. 
For $I \in \Ib_{r,n}$, we say that $I$ is \textit{of type $b$} if $|I \cap I_q|=b$, where 
\[q=q(I) := \min \{ t : I \cap I_t \neq \emptyset\}.\] 

We observe the following: 
\begin{Proposition}\label{prop:bdmf}
Let $I=\{i_1,\ldots,i_r\} \in \Ib_{r,n}$, where $1 \leq i_1 < \cdots < i_r \leq n$. 
Assume that $I \cap I_{q(I)}=\{i_1,\ldots,i_b\}$. 
Then $\Lambda_{\ab,\ell}$ sends $I$ as follows: 
\begin{align}\label{eq:BDMF}\Lambda_{\ab,\ell} (I) = 
\begin{cases}
(1 \ 2 \ \cdots \ \ell),\; &\mbox{if $I$ is of type }1, \\
(2 \ \cdots \ \ell),\; &\mbox{if $I$ is of type }2, \\
\quad\vdots \\
\quad\vdots \\
(\ell-1 \ \ell),\; &\mbox{if $I$ is of type } \ell-1, \\
 {\rm id}, &\mbox{otherwise}. 
\end{cases}
\end{align}
In particular, the $\ell$-th entry of $T_I$ is equal to $i_b$ and the other entries of $T_I$ are arrayed in the increasing order. 
\end{Proposition}
\begin{proof}
Among the entries of the $\ell$-th row of $M_{\ab,\ell}$, the right-most entry (i.e. the $b$-th entry) of $I_{q(I)}=[\alpha_q] \setminus [\alpha_{q-1}]$ 
is the smallest for the weights of the monomials in $\det(X_I)$, so we must choose it to make the weight as small as possible. 
Since the other rows are arranged in the descreasing order, we obtain the desired permutation. 
\end{proof}

\begin{Remark}
(a) When $\ab=(n)$, i.e., $s=1$, 
the corresponding block diagonal matching field is so-called the {\em diagonal matching field} (see \cite[Example 1.3]{sturmfels1993maximal}). 
This actually gives rise to the Gelfand-Tsetlin degeneration (see \cite[Section 14]{miller2005combinatorial}).

(b) The terminology ``$s$-block diagonal'' is used in \cite{higashitani2022quadratic}, which is the special case of Definition~\ref{def:block} with $\ell=2$. 
\end{Remark}

\begin{Example}\label{ex:block}
If $r=4$, $n=9$ and $\ab=(2,2,3,2)$ and $\ell=3$, then 
\[
M_{\ab,\ell}=
\begin{pmatrix}
0 & 0 & 0 & 0 & 0 & 0 & 0 & 0 & 0\\
9 &8 &7 &6 &5 &4 &3 &2 &1 \\
200 & 100 & 400 & 300 & 700 & 600 & 500 & 900 & 800 \\
90000 &80000 &70000 &60000 &50000 &40000 &30000 &20000 &10000 
\end{pmatrix}.
\]
For example, the matching field $\Lambda_{\ab,3}$ sends some elements of $\Ib_{4,9}$ as follows: 
\begin{align*}
1234 \mapsto (2 \ 3), \;\; 1349 \mapsto (1 \ 2 \ 3), \;\; 5678 \mapsto {\rm id}. 
\end{align*}
\end{Example}

\begin{Example}[Continuation of Example~\ref{ex:block}]
Work with the same $\Lambda_{\ab,3}$ as in Example~\ref{ex:block}. 
For $I=1234$, we see from Proposition~\ref{prop:bdmf} that 
the $3$rd (i.e. $\ell$-th) entry of $T_I$ is $2$, and $1,3,4$ are arrayed in the increasing order in the remaining boxes. 
Namely, we have $T_{1234}=\begin{ytableau}1 \\ 3 \\ 2 \\ 4\end{ytableau}$. 
Similarly, we have $T_{1349}=\begin{ytableau}3 \\ 4 \\ 1 \\ 9\end{ytableau}$, $T_{5678}=\begin{ytableau}5 \\ 6 \\ 7 \\ 8\end{ytableau}$, 
$T_{5689}=\begin{ytableau}5 \\ 8 \\ 6 \\ 9\end{ytableau}$, and so on. 
\end{Example}


Moreover, we have the following: 
\begin{Proposition}\label{prop:a1}
Let $\ab=(a_1,\ldots,a_s)$ and assume that $a_1 \leq \ell$. Let $\ab'=(1,a_1-1,a_2,\ldots,a_s)$. 
Then $\Lambda_{\ab,\ell}$ provides a SAGBI basis for $\Ac_{r,n}$ if and only if so does $\Lambda_{\ab',\ell}$.
\end{Proposition}
\begin{proof}
Let $\theta=(1 \, 2 \, \cdots \, a_1) \in \Sf_n$. This induces the permutation of columns, i.e., the variables of $\KK[\xb]$. 
Then for any $I \in \Ib_{r,n}$, we see that $\ini_{M_{\ab',\ell}}(\det(\xb_{\theta(I)}))=\theta(\ini_{M_{\ab,\ell}}(\xb_I))$ 
by definition of $\Lambda_{\ab,\ell}$ and $\Lambda_{\ab',\ell}$. Hence, the assertion follows from Remark~\ref{rem:permute}. 
\end{proof}

\subsection{New initial monomials}\label{subsec:initial}

For the proofs of the theorems given in Sections~\ref{sec:notSAGBI} and \ref{sec:yesSAGBI}, 
we discuss the initial monomials of $\psi(P_IP_J-P_{I'}P_{J'})$ with respect to a coherent matching field $\Lambda$ in terms of tableaux. 

We say that a tableau $T'$ is obtained from a tableau $T$ by a \textit{vertical swap} if $T'$ is obtained by swapping two boxes 
which are placed in the same column. 
For example, $\begin{ytableau} 1 &3 \\ 3 &4 \\ 2 &5 \\ 4 &6\end{ytableau}$ can be obtained 
from $\begin{ytableau} 1 &3 \\ 2 &4 \\ 3 &5 \\ 4 &6\end{ytableau}$ by a vertical swap. 
In this case, we swap two adjacent boxes. 
For example, $\begin{ytableau} 1 &3 \\ 3 &4 \\ 4 &5 \\ 2 &6\end{ytableau}$ can be obtained 
from $\begin{ytableau} 2 &3 \\ 3 &4 \\ 4 &5 \\ 1 &6\end{ytableau}$ by a vertical swap. 
In this case, the swapped boxes are not vertically adjacent, but it could also happen in the discussions below.

Regarding the monomials appearing in $\psi(P_I)$ for $I \in \Ib_{r,n}$ with respect to the block diagonal matching field $\Lambda_{\ab,\ell}$, 
we can observe that the monomial $\xb_{\sigma(I)}$ cannot be the second lowest monomial  
if the corresponding tableaux $T_{\sigma(I)}$ is obtained by at least two vertical swaps from $T_I$. This follows from by definition of $M_{\ab,\ell}$. 
This implies that for $I,J,I',J' \in \Ib_{r,n}$ such that $T_{I,J}$ and $T_{I',J'}$ are row-wise equal, 
the tableau corresponding to the initial monomial of $\psi(P_IP_J-P_{I'}P_{J'})$ can be obtained by a vertical swap of $T_{I,J}$ or $T_{I',J'}$. 

\begin{Example}
Consider the block diagonal matching field $\Lambda_{(2,4),2}$. 
Note that $\Lambda_{(2,4),2}$ is induced by $M_{(2,4),2}=\begin{pmatrix} 0 &0 &0 &0 &0 &0 \\ 2 &1 &6 &5 &4 &3 \\ 60 &50 &40 &30 &20 &10 \end{pmatrix}$. 
Let us discuss the initial monomial of $\psi(P_{135}P_{346}-P_{136}P_{345})$ with respect to $\Lambda_{(2,4),2}$. 
Here, $T_{135,346}$ and $T_{136,345}$ look like $\begin{ytableau} 3 &3 \\ 1 &4 \\ 5 &6 \end{ytableau}\;\text{ and }\;\begin{ytableau} 3 &3 \\ 1 &4 \\ 6 &5 \end{ytableau},$
which are row-wise equal. This means that the initial monomials of $\psi(P_{135}P_{346})$ and $\psi(P_{136}P_{345})$ coincide. 

Hence, we would like to consider which monomial among $\psi(P_{135}P_{346})$ and $\psi(P_{136}P_{345})$ 
becomes a new initial monomial of $\psi(P_{135}P_{346}-P_{136}P_{345})$. 
As mentioned above, we see that a new initial monomial can be obtained by a vertical swap of $T_{135,346}$ or $T_{136,345}$. 
If we apply a vertical swap on the first and second rows of $T_{135}$, then the swapped tableaux still become row-wise equal. 
Namely, this vertical swap does not give us a new initial monomial. 
Similarly, we can observe that any vertical swap on the first and second rows is not applicable. Hence, the monomials 
\[\begin{ytableau} 3 &3 \\ 1 &6 \\ 5 &4 \end{ytableau}, \;\; \begin{ytableau} 3 &3 \\ 1 &5 \\ 6 &4 \end{ytableau}, \;\; 
\begin{ytableau} 3 &3 \\ 5 &4 \\ 1 &6 \end{ytableau} \;\text{ and }\;\begin{ytableau} 3 &3 \\ 6 &4 \\ 1 &5 \end{ytableau}\]
are the candidates as a new initial monomial, which are obtained by vertical swaps on the second and third rows. 
By considering to make the weight of the third row as small as possible, 
we see that $\begin{ytableau} 3 &3 \\ 1 &5 \\ 6 &4 \end{ytableau}$ becomes a new initial monomial. 
\end{Example}
\begin{Example}
Let us discuss the new initial monomial of $\psi(P_{2678}P_{4567}-P_{5678}P_{2467})$ with respect to $\Lambda_{(3,2,3),4}$ on $\Gr(4,8)$. 
Then $T_{2678,4567}$ and $T_{5678,2467}$ look like $\ytableausetup{boxsize=1em}
\begin{ytableau} 6 & 4  \\ 7 & 6 \\ 8 & 7 \\ 2 & 5 \end{ytableau}$ and $\begin{ytableau} 6 & 4  \\ 7 & 6 \\ 8 & 7 \\ 5 & 2  \end{ytableau}$, respectively.  
Then the initial monomial corresponds to $\begin{ytableau} 6 & 5  \\ 7 & 6 \\ 8 & 7 \\ 2 &4 \end{ytableau}$. 
(The reason why this becomes the initial monomial will be explained in Example~\ref{ex:t=l}.) 
\end{Example}


\medskip


\section{$(\ab,\ell)$-block diagonal matching fields which do not provide a SAGBI basis}\label{sec:notSAGBI}


The goal of this section is to prove the following: 
\begin{Theorem}\label{thm:onlyif}
Work with the same notation as in Definition~\ref{def:block}. 
Assume that $a_s \geq 2$ and $\ell \geq 4$. If one of the following conditions is satisfied, 
then $(\ab,\ell)$-block diagonal matching field does not provide a SAGBI basis: 
\begin{itemize}
\item $a_1 \geq 5$; 
\item there is $2 \leq i \leq s-1$ with $a_i \geq 4$ and $r+2 \leq \sum_{u=i}^s a_u$. 
\end{itemize}
\end{Theorem}

Before giving a proof of this theorem, let us see the statement by some examples. 
\begin{Example}\label{ex:624}
A block diagonal matching field $\Lambda_{(6,2),4}$ of $\Gr(4,8)$ does not provide a SAGBI basis. 
\begin{proof}
Consider the initial monomial of $\psi(P_{1237}P_{4568} - P_{1238}P_{4567})$. 
Since $1237, 4568, 1238$ and $4567$ are of type $3$, we have the tableaux description: 
$\begin{ytableau} 1 & 4  \\ 2 & 5 \\ 7 & 8 \\ 3 &6 \end{ytableau}-\begin{ytableau} 1 & 4  \\ 2 & 5 \\ 8 & 7 \\ 3 &6 \end{ytableau}$. 
Note that those are row-wise equal. 
By applying a similar discussion in Subsection~\ref{subsec:initial}, we see that the initial monomial of $\psi(P_{1237}P_{4568} - P_{1238}P_{4567})$ corresponds to 
\begin{align}\label{eq:hanrei}\begin{ytableau} 1 & 4  \\ 2 & 7 \\ 8 & 5 \\ 3 &6 \end{ytableau}\end{align} 

Next, we show that the monomial \eqref{eq:hanrei} cannot be written as a product of the initial monomials of two Pl\"ucker variables. 
On the contrary, suppose that there are $I,I' \in \Ib_{4,8}$ such that $T_{I,I'}$ is row-wise equal to \eqref{eq:hanrei}. 
Let $\ytableausetup{boxsize=1.3em}
T_I=\begin{ytableau} i_1 \\ i_2 \\ i_3 \\ i_4 \end{ytableau}$ and $T_{I'}=\begin{ytableau} i_1' \\ i_2' \\ i_3' \\ i_4' \end{ytableau}$. 
Without loss of generality, we may assume that $i_3=5$. 
In the case of $i_2=7$, even if $i_4=3$ or $i_4=6$, we see that $\Lambda(I) \not\in \{\id, (3 \; 4), (2 \; 3 \; 4), (1 \; 2 \; 3 \; 4)\}$, 
a contradiction to Proposition~\ref{prop:bdmf} (see \eqref{eq:BDMF}). 
Thus, $i_2=2$. Similarly, we also obtain that $i_1=1$. Hence, $i_4=6$, otherwise $I=\{1,2,3,5\}$ and $\Lambda(I)=(3 \; 4)$, a contradiction to \eqref{eq:BDMF}. 
Hence, $I=\{1,2,5,6\}$, and so $I'=\{3,4,7,8\}$. Then $\Lambda(I')=(2 \; 3 \; 4)$, a contradiction to the form $T_{I'}$. 

Therefore, there are no $I,I' \in \Ib_{4,8}$ such that $T_{I,I'}$ is row-wise equal to \eqref{eq:hanrei}. 
\end{proof}
\end{Example}
\begin{Example}\label{ex:524}
A block diagonal matching field $\Lambda_{(5,2),4}$ of $\Gr(4,7)$ does not provide a SAGBI basis. 
\begin{proof}
Consider the initial monomial of $\psi(P_{1236}P_{3457} - P_{1237}P_{3456})$. 
Since $1236, 3457, 1237$ and $3456$ are of type $3$, we have the tableaux description $\ytableausetup{boxsize=1em}
\begin{ytableau} 1 & 3  \\ 2 & 4 \\ 6 & 7 \\ 3 &5 \end{ytableau} - \begin{ytableau} 1 & 3  \\ 2 & 4 \\ 7 & 6 \\ 3 &5 \end{ytableau}$, 
which are row-wise equal. 
Then the initial monomial of $\psi(P_{1236}P_{3457} - P_{1237}P_{3456})$ corresponds to $\begin{ytableau} 1 & 3  \\ 2 & 6 \\ 7 & 4 \\ 3 &5 \end{ytableau}$.  
If there are $I,I' \in \Ib_{4,7}$ such that $T_{I,I'}$ is row-wise equal to this tableau, where we let the third entry of $I$ is $4$, 
then we see that $T_I=\begin{ytableau} 1 \\ 2 \\ 4 \\ 5\end{ytableau}$, so $T_{I'}$ must be $\begin{ytableau} 3 \\ 6 \\ 7 \\ 3\end{ytableau}$. 
This is a contradiction since $3$ appears twice. 
\end{proof}
\end{Example}

By the similar idea to Example~\ref{ex:524}, we can prove Theorem~\ref{thm:onlyif}. 
\begin{proof}[Proof of Theorem~\ref{thm:onlyif}]
We prove the case of $\ell=r$. (Other cases can be similarly proved.) 

\noindent
\underline{$a_1 \geq 5$}: 
Let $a_1 = a$. 

Let us discuss the case $a \leq r$. 
Since we always assume $n-r \geq 3$, i.e., $n \geq r+3$ (see Remark~\ref{rem:assume}), we have $\sum_{i=2}^s a_i=n-a \geq r+3-a$. 
Hence, there are at least $(r+3-a)$ columns from the second block after the first $a$ columns. Let 
\begin{align*}
I&=[r+3] \setminus \{a-1,a,a+2\}, \;\; J=[r+3] \setminus \{1,2,a+1\}, \\
I'&=(I \setminus \{a+1\}) \cup \{a+2\} \;\;\text{and}\;\; J'=(J \setminus \{a+2\}) \cup \{a+1\}. 
\end{align*}
Then all of $I,J,I',J'$ are of type $(a-2)$, and $T_{I,J}$ and $T_{I',J'}$ look as follows: 
$$\ytableausetup{boxsize=2.2em}
T_{I,J}=\begin{ytableau} 1 &3 \\ \vdots &\vdots \\ a-3 &a-1 \\ a+1 &a+2 \\ a+3 &a+3 \\ \vdots &\vdots \\ a-2 &a\end{ytableau} \;\; \text{ and } \;\;
T_{I',J'}=\begin{ytableau} 1 &3 \\ \vdots &\vdots \\ a-3 &a-1 \\ a+2 &a+1 \\ a+3 &a+3 \\ \vdots &\vdots \\ a-2 &a\end{ytableau}. 
$$
Note that $T_{I,J}$ and $T_{I',J'}$ are row-wise equal, the only difference is the $(a-2)$-th row. 
Actually, the rows after the $(a-2)$-th row of both $T_{I,J}$ and $T_{I',J'}$ are not involved in the following discussions. 
Consider the initial monomial of $\psi(P_IP_J-P_{I'}P_{J'})$. 
By a similar discussion of Subsection~\ref{subsec:initial}, we see that a crucial vertical swap is the one on the $(a-3)$-th and $(a-2)$-th rows. 
Namely, the new initial monomial appears by applying such a vertical swap. 
Since it is better that $a+2$ in the $(a-2)$-th row stays the same row to make the weight smaller, we see that our desired vertical swap is 
the one replacing $a-1$ and $a+1$ of $T_{I',J'}$, i.e., the initial monomial of $\psi(P_IP_J-P_{I'}P_{J'})$ corresponds to the tableau 
$\begin{ytableau} 1 &3 \\ \vdots &\vdots \\ a-3 &a+1 \\ a+2 &a-1 \\ a+3 &a+3 \\ \vdots &\vdots \\ a-2 &a\end{ytableau}$.
Suppose that $I'',J'' \in \Ib_{r,n}$ such that $T_{I'',J''}$ is row-wise equal to this tableau. 
Without loss of generality, we may assume that the $(a-2)$-th entry of $T_{I''}$ is $a-1$. 
Then, for each $1 \leq i \leq a-3$, we see that the $i$-th entry of $T_{I''}$ should be $i$ due to \eqref{eq:BDMF}. 
Thus, for each $1 \leq j \leq r-1$, the $j$-th entry of $T_{J''}$ should be $j+2$. 
In particular, by the assumption $a \geq 5$, both $a-2$ and $a$ are included in the first $(r-1)$ rows of $T_{J''}$. 
This leads a contradiction since the $r$-th row should be either $a-2$ or $a$. 

In the case $a>r$, we may set $I,J,I',J'$ as follows: 
\begin{align*}
I&=\{1,2,\ldots,r-1\} \cup \{n-1\}, \;\; J=\{3,4,\ldots,r+1\} \cup \{n\}, \\
I'&=\{1,2,\ldots,r-1\} \cup \{n\} \;\text{ and }\; J'=\{3,4,\ldots,r+1\} \cup \{n-1\}. 
\end{align*}
Then $I$, $I'$, $J$ and $J'$ are of type $(r-1)$, and $T_{I,J}$ and $T_{I',J'}$ looks as follows: 
$$\ytableausetup{boxsize=2.1em}
T_{I,J}=\begin{ytableau} 1 &3 \\ \vdots &\vdots \\ r-2 &r \\ n-1 &n \\ r-1 &r+1\end{ytableau} \;\; \text{ and } \;\;
T_{I',J'}=\begin{ytableau} 1 &3 \\ \vdots &\vdots \\ r-2 &r \\ n &n-1 \\ r-1 &r+1\end{ytableau}, 
$$
where we remark that $r<n-1$ since $n-r \geq 3$ is assumed, 
and the initial monomial of $\psi(P_IP_J-P_{I'}P_{J'})$ corresponds to $\begin{ytableau}1 &3 \\ \vdots &\vdots \\ r-2 &n-1 \\ n &r \\ r-1 &r+1\end{ytableau}$. 
If there were $I'',J'' \in \Ib_{r,n}$ such that $T_{I'',J''}$ is row-wise equal to this, where we let $n-1 \in I''$, 
then $I''$ should be of type either $(r-1)$ or $(r-2)$, but this is a contradiction to the possible form of $T_{I''}$. 

\medskip

\noindent
\underline{There is $2 \leq i \leq s-1$ such that $a_i \geq 4$ and $r+2 \leq \sum_{u=i+1}^sa_u$}: 
Let $i$ be the minimum integer such that $a_i \geq 4$ and $r+2 \leq \sum_{u=i}^sa_u$ are satisfied. 
Let $p_1<p_2<\ldots<p_{a_i}$ be the integers with $\{p_j : j=1,\ldots,a_i\}=\{\sum_{u=1}^{i-1}a_u+1,\ldots,\sum_{u=1}^ia_u\}$,  
and let $q_1<q_2<\ldots$ be the ones with $\{q_j : j \geq 1\}=\{\sum_{u=1}^ia_u+1,\ldots,n\}$. 
By our assumption, we see that $|\{q_j : j \geq 1\}| \geq r+2$. 

Let us discuss the case $a_i \leq r$. Let
\begin{align*}
I&=\{p_1,\ldots,p_{a_i-1}\} \cup \{q_1\} \cup \{q_3,\ldots,q_{r+2-a_i}\}, \\
J&=\{1\} \cup \{p_3,\ldots,p_{a_i}\} \cup \{q_2,\ldots,q_{r+2-a_i}\}, \\ 
I'&=(I \setminus \{q_1\}) \cup \{q_2\} \;\;\text{ and }\;\; J'=(J \setminus \{q_2\}) \cup \{q_1\}. 
\end{align*}
Then $I$ and $I'$ are of type $(a_i-1)$, while $J$ and $J'$ are of type $1$. Hence, $T_{I,J}$ and $T_{I',J'}$ look as follows: 
\begin{align*}\ytableausetup{boxsize=2.3em}
T_{I,J}=\begin{ytableau} p_1 &p_3 \\ \vdots &\vdots \\ p_{a_i-2} &p_{a_i} \\ q_1 &q_2 \\ q_3 &q_3 \\ \vdots &\vdots \\ p_{a_i-1} &1 \end{ytableau} \;\; \text{ and } \;\;
T_{I',J'}=\begin{ytableau} p_1 &p_3 \\ \vdots &\vdots \\ p_{a_i-2} &p_{a_i} \\ q_2 &q_1 \\ q_3 &q_3 \\ \vdots &\vdots \\ p_{a_i-1} &1 \end{ytableau}. 
\end{align*}
Thus, the initial monomial of $\psi(P_IP_J-P_{I'}P_{J'})$ corresponds to 
$\begin{ytableau} p_1 &p_3 \\ \vdots &\vdots \\ p_{a_i-2} &q_1 \\ q_2 &p_{a_i} \\ q_3 &q_3 \\ \vdots &\vdots \\ p_{a_i-1} &1 \end{ytableau}$. 
If there were $I'',J'' \in \Ib_{r,n}$ such that $T_{I'',J''}$ is row-wise equal to this, where we let $p_{a_i} \in I''$, 
then the $j$-th row of $T_{I''}$ should be $p_j$ for each $1 \leq j \leq a_i-2$. Moreover, we see that the $r$-th row of $T_{I''}$ should be $1$. 
Then $p_{a_i-1}$ appears twice in $T_{J''}$, a contradiction. 

In the case $a_i>r$, we may set
\begin{align*}
I&=\{p_1,\ldots,p_{r-1}\} \cup \{q_1\}, \;\;J=\{1\} \cup \{p_3,\ldots,p_r\} \cup \{q_2\}, \\
I'&=(I \setminus \{q_1\}) \cup \{q_2\} \;\;\text{ and }\;\; J'=(J \setminus \{q_2\}) \cup \{q_1\} 
\end{align*}
and apply the similar discussion. 
\end{proof}

\medskip


\section{Toric degenerations associated to $(\ab,\ell)$-block diagonal matching fields}\label{sec:yesSAGBI}

\subsection{``The first'' part of Theorem~\ref{thm:main}}
In this section, we provide a new family of toric degenerations of $\Gr(r,n)$. We prove the first part of Theorem \ref{thm:main}. 
\begin{Theorem}\label{thm:if}
Let $\ab=(a_1,\ldots,a_s) \in \ZZ_{>0}^s$ satisfying $\sum_{i=1}^s a_i = n$ and $a_s \geq 2$, and let $\ell \geq 4$. 
Assume that $a_1 \leq 3$ and $a_i \leq 2$ for any $2 \leq i \leq s-1$. 
Then $(\ab,\ell)$-block diagonal matching field provides a SAGBI basis. 
\end{Theorem}
\begin{proof}
Thanks to Proposition~\ref{prop:a1}, we may also assume that $a_i \leq 2$ for any $1 \leq i \leq s-1$. 
This implies that any $I \in \Ib_{r,n}$ is always of type $1$ or $2$ unless $I \subset I_s$. 

Let $<$ be the monomial order associated to $\Lambda_{\ab,\ell}$.  
Our goal is to show that any monomial of $\ini_<(\Im(\psi)) \subset \ini_<(\Ac_{r,n})$ can be written as a product of $\ini_<(\psi(P_I))$'s. 

Take an arbitrary polynomial \[f=\sum_{\alpha} c_\alpha\Pb^\alpha \in \KK[P_I : I \in \Ib_{r,n}]=S,\] 
where $c_\alpha \in \KK$ is $0$ all but finitely many $\alpha$'s and $\Pb^\alpha$ stands for a monomial of $S$, 
and consider the initial monomial of $\psi(f)$ with respect to $<$. 
If \[\ini_<(\psi(f))=\max\{\ini_<(\psi(\Pb^\alpha)) : c_\alpha \neq 0\},\] i.e., 
there is a monomial $\Pb^{\alpha_0}$ of $S$ such that $\ini_<(\psi(f))=\ini_<(\psi(\Pb^{\alpha_0}))$, 
then we are done. 

In what follows, we assume that $$\ini_<(\psi(f))<\max\{\ini_<(\psi(\Pb^\alpha)) : c_\alpha \neq 0\}.$$ 
Then we see that there are two monomials $\Pb^\beta$ and $\Pb^\gamma$ in $f$ such that $$\ini_<(\psi(f))=\ini_<(\psi(\Pb^\beta-\epsilon\Pb^\gamma)),$$
where $\epsilon \in \{\pm 1\}$.
Let $T$ (resp. $T'$) be the tableau corresponding to $\ini_<(\psi(\Pb^\beta))$ (resp. $\ini_<(\psi(\Pb^\gamma))$). Then $T$ and $T'$ are row-wise equal.  
Assume that the first $(t-1)$ rows of $T$ and $T'$ are exactly the same and the $t$-th rows are different, where $t \geq 2$. 
We consider the tableau $\tilde{T}$ corresponding to the new initial monomial of $\psi(\Pb^\beta-\epsilon\Pb^\gamma)$. 
Then we see 
that $\tilde{T}$ can be obtained via a vertical swap of the $(t-1)$-th and $t$-th rows of $T$ or $T'$. 

\noindent
{\bf The first step}: First, we discuss the case of $t \neq 2,\ell,\ell+1$. 
By our assumption, any $I \in \Ib_{r,n}$ is of type $1$ or $2$ unless $I \subset I_s$. 
Hence, if $t \neq 2,\ell,\ell+1$, i.e., the vertical swap occurs at the rows not concerned with the first, and $\ell$-th rows, 
then we can deduce the discussion in the case of the diagonal matching field, i.e., we can ignore the $\ell$-th row. 

For example, let us consider the block diagonal matching field $\Lambda_{(2,2,6),5}$ and the new initial monomial corresponding to 
$\ytableausetup{boxsize=1em}
\begin{ytableau} 3 & 4  \\ 5 & 6 \\ 7 & 8 \\ 9 &10 \\ 1 &2 \end{ytableau}-\begin{ytableau} 3 & 4  \\ 5 & 6 \\ 7 & 8 \\ 10 &9 \\ 1 & 2  \end{ytableau}$, 
where this is the case with $\ell=5$ and $t=4$. 
Let us remove the $\ell$-th row from the considering tableaux, like $\ytableausetup{boxsize=1em}
\begin{ytableau} 3 & 4  \\ 5 & 6 \\ 7 & 8 \\ 9 &10 \end{ytableau}-\begin{ytableau} 3 & 4  \\ 5 & 6 \\ 7 & 8 \\ 10 &9 \end{ytableau}$. 
Then we see that this is already deduced to the case of the diagonal matching field for $\Gr(r-1,n)$, 
so there is a tableau which is row-wise equal corresponding to $\ini_<(\psi(\Pb^\delta))$ for some $\delta \in \ZZ_{\geq 0}^{n-1}$. 
In the example, we know that such tableau is $\begin{ytableau} 3 & 4  \\ 5 & 6 \\ 7 & 9 \\ 8 &10 \end{ytableau}$. 
Therefore, we can also obtain the tableau 
corresponding to $\ini_<(\psi(\Pb^\delta))$ for some $\delta \in \ZZ_{\geq 0}^n$, 
i.e., $\begin{ytableau} 3 & 4  \\ 5 & 6 \\ 7 & 9 \\ 8 &10 \\ 1 &2 \end{ytableau}$ in the example.

Hence, we may focus on the cases $t=2,\ell,\ell+1$. 
We set that the new initial monomial occurs from the following: 
$$\ytableausetup{boxsize=1.2em}
\begin{ytableau} p_1 & q_1  & r_1 & \cdots \\ p_2 & q_2  & r_2 & \cdots \\ \vdots & \vdots & \vdots & \cdots \\ p_t & q_t  & r_t & \cdots \\ 
\vdots & \vdots & \vdots & \cdots \\ p_\ell & q_\ell & r_\ell & \cdots \\ \vdots & \vdots & \vdots & \cdots \end{ytableau}-
\begin{ytableau} p_1 & q_1  & r_1 & \cdots \\ p_2 & q_2  & r_2 & \cdots \\ \vdots & \vdots & \vdots & \cdots \\ p_t' & q_t'  & r_t' & \cdots \\ 
\vdots & \vdots & \vdots & \cdots \\ p_\ell' & q_\ell' & r_\ell' & \cdots \\ \vdots & \vdots & \vdots & \cdots \end{ytableau}, $$
where those two tableaux are row-wise equal. In particular, for $t \leq j \leq r$, the contents $p_j',q_j',\ldots$ are just a rearrangement of $p_j,q_j,\ldots$. 
We assume that the vertical swap occurs at the column of $q_i$'s, i.e., the second column of the first tableau in this description.

\noindent
{\bf The second step}: We discuss the case of $t=2$. Namely, we discuss the initial monomial of a polynomial corresponding to 
$\ytableausetup{boxsize=1.2em} T-T'=
\begin{ytableau} p_1 & q_1  & r_1 & \cdots \\ p_2 & q_2  & r_2 & \cdots \\ p_3 &q_3 &r_3 &\cdots \\ \vdots & \vdots & \vdots & \cdots \\\end{ytableau}
-\begin{ytableau} p_1 & q_1  & r_1 & \cdots \\ q_2 & r_2 & \cdots & p_2 \\ p_3' &q_3' &r_3' &\cdots \\ \vdots & \vdots & \vdots & \cdots \\\end{ytableau}$
and assume that its corresponding tableau after the vertical swap looks as follows: 
\begin{align}\label{eq:tab_t=2}
\begin{ytableau} p_1 & q_2  & r_1 & \cdots \\ p_2 & q_1 & r_2 & \cdots \\ p_3' & q_3' &r_3' &\cdots \\ \vdots & \vdots & \vdots & \cdots \end{ytableau}
\end{align} 
Let us consider the tableau which looks as follows: 
\begin{align}\label{eq:tab_t=2new}\begin{ytableau} p_1 & q_2  & r_1 &\cdots \\ q_1 & r_2  & \cdots &p_2 \\ 
p_3' &q_3' &r_3' &\cdots \\ \vdots & \vdots & \vdots & \vdots \end{ytableau}\end{align} 
Note that \eqref{eq:tab_t=2} and \eqref{eq:tab_t=2new} are row-wise equal. 
In what follows, we claim that the monomial corresponding to \eqref{eq:tab_t=2} (this is the same as the monomial corresponding to \eqref{eq:tab_t=2new}) 
can be written as a product of $\ini_<(\psi(P_I))$'s. 
See Example~\ref{ex:t=2} for the explanation of this case. 

Let $I^{(1)} \in \Ib_{r,n}$ (resp. $I^{(2)} \in \Ib_{r,n}$) be such that $T_{I^{(1)}}$ (resp. $T_{I^{(2)}}$) coincides with the first (resp. second) column of $T'$. 
\begin{itemize}
\item The columns of \eqref{eq:tab_t=2new} all but the first two columns are completely the same as those of the tableaux corresponding to $T'$. 
Hence, those columns clearly arise from some initial monomials of certain $\psi(P_I)$'s. 
\item By the shape of \eqref{eq:tab_t=2}, we have $q_2<r_2$ (resp. $p_1<q_1$) 
since the weight of the monomial corresponding to \eqref{eq:tab_t=2} is lower than that corresponding to the one 
obtained from $T'$ (resp. $T$) by the vertical swap of the first and second entries of the second (resp. first) column. 
\item Assume that $q_2<p_\ell'$. Then $p_1<q_2<p_\ell'$. 
Thus $\{p_1,q_2,p_\ell'\} \subset I^{(1)}_{q(I^{(1)})}$ holds by definition of $\Lambda_{\ab,\ell}$. 
Hence, we see that the first column of \eqref{eq:tab_t=2new} corresponds to $\ini_<(\psi(P_{(I^{(1)} \setminus \{q_2\}) \cup \{q_1\}}))$. 
Note that $q_1<q_2<r_2<q_3'<\cdots$ hold. In the cases of both $\{q_1,q_\ell'\} \subset I^{(2)}_{q(I^{(2)})}$ and $\{q_1,q_\ell'\} \not\subset I^{(2)}_{q(I^{(2)})}$, 
we see that the second column corresponds to $\ini_<(\psi(P_{(I^{(2)} \setminus \{q_1\}) \cup \{q_2\}}))$.  
\item Assume that $q_2 \geq p_\ell'$. 
\begin{itemize}
\item If $\{p_1,q_2,p_\ell'\} \not\subset I^{(1)}_{q(I^{(1)})}$, i.e., 
either ``$I^{(1)}_{q(I^{(1)})}=\{p_\ell'\}$ and $p_1,q_2 \not\in I^{(1)}_{q(I^{(1)})}$'' or ``$I^{(1)}_{q(I^{(1)})}=\{p_1,p_\ell'\}$ and $q_2 \not\in I^{(1)}_{q(I^{(1)})}$'', 
then we see that the first column of \eqref{eq:tab_t=2new} corresponds to $\ini_<(\psi(P_{(I^{(1)} \setminus \{q_2\}) \cup \{q_1\}}))$. 
We also see that the second column also corresponds to $\ini_<(\psi(P_{(I^{(2)} \setminus \{q_1\}) \cup \{q_2\}}))$.  
\item Let $\{p_1,q_2,p_\ell'\} \subset I^{(1)}_{q(I^{(1)})}$. 
In this case, we swap the $\ell$-th entries of the first and second columns of \eqref{eq:tab_t=2new}, i.e., $p_\ell'$ and $q_\ell'$. 
Then we see that 
$\begin{ytableau} p_1 \\ q_1 \\ p_3' \\ \vdots \\ q_\ell' \\ \vdots \end{ytableau}$ corresponds to $\ini_<(\psi(P_{(I^{(1)} \setminus \{q_2,p_\ell'\}) \cup\{q_1,q_\ell'\}}))$. 
In fact, if $q_\ell'<q_1$, since $p_1<q_1<q_2$, we have $q_\ell'<p_1$; if $q_\ell' > q_1$, then $\{p_1,q_1,q_2,p_\ell',q_\ell'\} \subset I^{(1)}_{q(I^{(1)})}$ 
(though some of them might be equal). 

Moreover, since $p_\ell' (\leq p_1)<q_2<r_2<q_3'<\cdots$, we see that the second column $\begin{ytableau} q_2 \\ r_2 \\ q_3' \\ \vdots \\ p_\ell' \\ \vdots \end{ytableau}$
corresponds to $\ini_<(\psi(P_{(I^{(2)} \setminus \{q_1,q_\ell'\}) \cup \{q_2,p_\ell'\}}))$. 
\end{itemize}
\end{itemize}

{\bf The third step}: Next, we discuss the case of $t=\ell$. Namely, we discuss the initial monomial corresponding to 
$T-T'=\ytableausetup{boxsize=2em}
\begin{ytableau} p_1 & q_1  & r_1 & \cdots \\ \vdots & \vdots & \vdots & \cdots \\ p_{\ell-1} & q_{\ell-1} & r_{\ell-1} & \cdots \\ 
p_\ell & q_\ell & r_\ell &\cdots \\ \vdots & \vdots & \vdots & \cdots \end{ytableau}
-\begin{ytableau} p_1 & q_1  & r_1 & \cdots \\ \vdots & \vdots & \vdots & \cdots \\ p_{\ell-1} & q_{\ell-1} & r_{\ell-1} & \cdots \\ 
q_\ell & r_\ell &\cdots &p_\ell \\ \vdots & \vdots & \vdots & \cdots \end{ytableau}$.  

Let $I=\{q_1,\ldots,q_r\} \in \Ib_{r,n}$ be corresponding to the second column of $T$ and let 
\[j=\begin{cases}
i-1 &\text{ if $I$ is of type $i$ with $2 \leq i \leq \ell-1$}, \\
|I_{t'} \cap I|+1 &\text{ if $I$ is of type $1$}, \\
\ell-1 &\text{ otherwise}, 
\end{cases}\]
where $t'=\min\{t \in [r] \setminus \{q(I)\} : I_t \cap I \neq \emptyset\}$. 
Since the second smallest entry in the $\ell$-th row among the $q_i$-th entries is the $q_j$-th one, 
the initial monomial of $T-T'$ corresponds to the tableau obtained by replacing $q_\ell$ and $q_j$ of $T$. 
By considering the tableau $\ytableausetup{boxsize=2em}
\begin{ytableau} p_1 & q_1  & r_1 & \cdots \\ \vdots & \vdots & \vdots & \cdots \\ 
p_j & q_\ell & r_j & \cdots \\ \vdots & \vdots & \vdots & \cdots \\ 
p_{\ell-1} & q_{\ell-1} & r_{\ell-1} & \cdots \\ q_j & r_\ell & \cdots &p_\ell \\ \vdots & \vdots & \vdots & \cdots \end{ytableau}$,
which is row-wise equal to the tableau corresponding to the initial monomial, 
we see by a similar discussion as above that each of the columns can be obtained by certain initial monomials. 

See Example~\ref{ex:t=l}. 

\noindent
{\bf The fourth step}: Finally, we discuss the case of $t=\ell+1$. 
In this case, the associated tableau to the initial monomial of $T-T'$ 
can be obtained by replacing $q_{\ell+1}$ and $q_{\ell-1}$ or $q_\ell$. 
(This is because the second smallest weight corresponds to either $q_{\ell-1}$ or $q_\ell$.) 
\begin{itemize}
\item If it is $q_\ell$, then 
$\ytableausetup{boxsize=2em}
\begin{ytableau} p_1 & q_1  & r_1 & \cdots \\ \vdots & \vdots & \vdots & \cdots \\ 
p_{\ell-1} & q_{\ell-1} & r_{\ell-1} & \cdots \\ p_\ell & q_{\ell+1} & r_\ell &\cdots \\ q_\ell &r_{\ell+1} &\cdots &p_{\ell+1} \\ 
\vdots & \vdots & \vdots & \cdots \end{ytableau}$ corresponds to the initial monomial and each of the columns can be obtained by certain initial monomials. 
\item If it is $q_{\ell-1}$, then $\ytableausetup{boxsize=2em}
\begin{ytableau} p_1 & q_1  & r_1 & \cdots \\ \vdots & \vdots & \vdots & \cdots \\ 
p_{\ell-1} & q_{\ell+1} & r_{\ell-1} & \cdots \\ p_\ell & q_\ell & r_\ell &\cdots \\ q_{\ell-1} &r_{\ell+1} &\cdots &p_{\ell+1} \\ 
\vdots & \vdots & \vdots & \cdots \end{ytableau}$ corresponds to the initial monomial and each of the columns can be obtained by certain initial monomials. 
\end{itemize}

See Example~\ref{ex:t=l+1}. 
\end{proof}

\begin{Example}\label{ex:t=2} 
Let us consider the new initial monomial of $\psi(P_{1246}P_{2358}-P_{1256}P_{2348})$ with respect to $\Lambda_{(3,2,3),4}$ on $\Gr(4,8)$,  
whose tableau description is $\ytableausetup{boxsize=1em}
\begin{ytableau} 1 & 2  \\ 4 & 5 \\ 6 & 8 \\ 2 &3 \end{ytableau}-\begin{ytableau} 1 & 2  \\ 5 & 4 \\ 6 & 8 \\ 2 & 3  \end{ytableau}$. 
Then the initial monomial corresponds to $\begin{ytableau} 1 & 4  \\ 5 & 2 \\ 6 & 8 \\ 2 &3 \end{ytableau}$. 
By shifting this tableau as described in the second step of the above proof, we obtain the tableau 
$\begin{ytableau} 1 & 4  \\ 2 & 5 \\ 6 & 8 \\ 2 &3 \end{ytableau}$. 
Note that the first column is not of the form of the initial monomial.  
Hence, by shifting the $\ell$-th row, i.e., the fourth row, we obtain that $\begin{ytableau} 1 & 4  \\ 2 & 5 \\ 6 & 8 \\ 3 &2 \end{ytableau}$. 
We see that the first (resp. second) column of this tableau is of the form of the initial monomial of $\psi(P_{1236})$ (resp. $\psi(P_{2458})$). 
\end{Example}


\begin{Example}\label{ex:t=l} 
Let us consider the new initial monomial of $\psi(P_{2678}P_{4567}-P_{5678}P_{2467})$ with respect to $\Lambda_{(3,2,3),4}$ on $\Gr(4,8)$, 
whose tableau description is $\ytableausetup{boxsize=1em}
\begin{ytableau} 6 & 4  \\ 7 & 6 \\ 8 & 7 \\ 2 & 5 \end{ytableau}-\begin{ytableau} 6 & 4  \\ 7 & 6 \\ 8 & 7 \\ 5 & 2  \end{ytableau}$. 
Then the initial monomial corresponds to $\begin{ytableau} 6 & 5  \\ 7 & 6 \\ 8 & 7 \\ 2 &4 \end{ytableau}$. 
By shifting this tableau as described in the second step of the above proof, we obtain the tableau 
$\begin{ytableau} 6 & 5  \\ 7 & 6 \\ 8 & 7 \\ 4 &2 \end{ytableau}$. 
We see that the first (resp. second) column of this tableau is of the form of the initial monomial of $\psi(P_{4678})$ (resp. $\psi(P_{2567})$). 
\end{Example}

\begin{Example}\label{ex:t=l+1}
Let us consider the new initial monomial of $\psi(P_{12569}P_{34578}-P_{12568}P_{34579})$ with respect to $\Lambda_{(2,2,2,4),4}$ on $\Gr(5,10)$, 
whose tableau description is $\ytableausetup{boxsize=1em}
\begin{ytableau} 1 & 3  \\ 5 & 5 \\ 6 & 7 \\ 2 & 4 \\ 9 & 8 \end{ytableau}-\begin{ytableau} 1 & 3  \\ 5 & 5 \\ 6 & 7 \\ 2 & 4 \\ 8 &9  \end{ytableau}$. 
Then the initial monomial corresponds to $\begin{ytableau} 1 & 3  \\ 5 & 5 \\ 6 & 8 \\ 2 & 4 \\ 9 & 7 \end{ytableau}$. 
By shifting this tableau as described in the second step of the above proof, we obtain the tableau 
$\begin{ytableau} 1 & 3  \\ 5 & 5\\ 6 & 8 \\ 2 & 4\\ 7 & 9 \end{ytableau}$. 
We see that the first (resp. second) column of this tableau is of the form of the initial monomial of $\psi(P_{12567})$ (resp. $\psi(P_{34589})$). 
\end{Example}

We can prove the following: 
\begin{Lemma}\label{cor:if}
Work with the same notation and the assumption as that of Theorem~\ref{thm:if}. 
Assume that for any $2 \leq i \leq s-1$ satisfying $a_i \geq 3$, $r \geq \sum_{u=i}^sa_u$ holds. 
Then $(\ab,\ell)$-block diagonal matching field provides a SAGBI basis. 
\end{Lemma}
\begin{proof}
For $i$ with $a_i \geq 3$, if $r>\sum_{u=i}^s a_u$, then $q(I)$ is never equals to $i$ for any $I\in\mathbf{I}_{r,n}$. If $r = \sum_{u=i}^s a_u$, then $I=I_i \cup \cdots \cup I_s$ by definition. This implies that for any tableau $T$ consisting of multiple columns one of which corresponds to $I$, any other tableau $T'$ which is row-wise equal to $T$ should contain the column corresponding to $I$. Hence, we can exclude such case. Namely, it is sufficient to consider $I$ of type $1$ or $2$. We come back to the proof of Theorem 5.1. 
\end{proof}

By using Theorems \ref{thm:onlyif}, \ref{thm:if}, Lemma \ref{cor:if}, we obtain Theorem \ref{thm:main}.

\subsection{Proof of the case $\ell=3$}
Next, we prove Theorem \ref{thm:l=3}.

\begin{proof}[Proof of Theorem \ref{thm:l=3}]
Work with the same notation as that of Theorem \ref{thm:if}.

First, we note that if $I \in \Ib_{r,n}$ is not of type $1$ or $2$ then $\Lambda_{\ab,\ell}$ sends $I$ to $\mathrm{id}$.

Since $\ell =3$, if $t \neq 2,\ell,\ell+1$, i.e., the vertical swap occurs at the rows not concerned with the first, second, and $\ell$-th rows, 
then we can deduce the discussion in the case of the diagonal matching field, i.e., we can ignore the $\ell$-th row. 

The rest of the argument is identical to the proof of the Theorem \ref{thm:if} except for the fourth step.
We modify it as follows.

\begin{itemize}
\item If it is $q_\ell$, then 
$\ytableausetup{boxsize=2em}
\begin{ytableau} p_1 & q_1  & r_1 & \cdots \\ 
p_{2} & q_{4} & r_{2} & \cdots \\ p_3 & q_{2} & r_3 &\cdots \\ q_3 &r_{4} &\cdots &p_{4} \\ 
\vdots & \vdots & \vdots & \cdots \end{ytableau}$ 
or $\begin{ytableau} p_1 & q_2  & r_1 & \cdots \\ 
p_{2} & q_{4} & r_{2} & \cdots \\ p_3 & q_{1} & r_3 &\cdots \\ q_3 &r_{4} &\cdots &p_{4} \\ 
\vdots & \vdots & \vdots & \cdots \end{ytableau}$
or $\begin{ytableau} p_1 & q_1  & r_1 & \cdots \\ 
p_{2} & q_{2} & r_{2} & \cdots \\ p_3 & q_{4} & r_3 &\cdots \\ q_3 &r_{4} &\cdots &p_{4} \\ 
\vdots & \vdots & \vdots & \cdots \end{ytableau}$
corresponds to the initial monomial and each of the columns can be obtained by certain initial monomials. We note that only the last pattern of this proof appears in the proof of Theorem \ref{thm:if}.
The reason for this is that $a_i\leq 2$ holds in the previous Theorem.

\item If it is $q_{\ell-1}$, then $\ytableausetup{boxsize=2em}
\begin{ytableau} p_1 & q_1  & r_1 & \cdots \\ 
p_{2} & q_{4} & r_{2} & \cdots \\ p_3 & q_3 & r_3 &\cdots \\ q_{2} &r_{4} &\cdots &p_{4} \\ 
\vdots & \vdots & \vdots & \cdots \end{ytableau}$ corresponds to the initial monomial and each of the columns can be obtained by certain initial monomials. 
\end{itemize}

\end{proof}

\subsection{Remaining cases}
\label{open}

What is not discussed in Theorem \ref{thm:main} is the following three cases.

\begin{itemize}
    \item[(i)] $a_1=4\text{ and } r+1\geq \sum_{u=i}^s a_u \text{ for all } i \text{ with } a_i \geq 4$;
    \item[(ii)] $a_1\leq 3 \text{ and } r+1=\sum_{u=i}^s a_u \text{ for all } i \text{ with } a_i\geq 4$; 
    \item[(iii)] $a_1\leq 4 \text{ and there exists } 2\leq i \leq s-1 \text{ such that }a_{i'} \leq 2 \text{ for } 2\leq i'< i, a_i=3 \text{ and }r+2\leq \sum_{u=i}^s a_u$, and $r+1 \geq \sum_{u=j}^s a_u$ is satisfied for any $j$ with $2 \leq j \leq s-1$ and $a_j \geq 4$.
\end{itemize}

By Proposition \ref{prop:a1}, the case (i) can be reduced to the case (iii). 
Since we assume type $<3$ in the proof of Theorem \ref{thm:if}, we cannot handle the case (iii), which is the case where the matching field does not have a type of $\geq 4$ and has type $3$. 
The case (ii) remains open.

\bibliography{Biblio}
\bibliographystyle{abbrv}


{

\end{document}